\documentclass[leqno,12pt,a4paper]{amsart}

\usepackage{amsmath,amssymb,amsthm}
\usepackage{enumerate}
\usepackage{graphics,color}
\usepackage{natbib}
\usepackage{hhline}

\newtheorem{thm}{Theorem}[section]
\newtheorem{cor}[thm]{Corollary}
\newtheorem{lem}[thm]{Lemma}
\newtheorem{prop}[thm]{Proposition}
\theoremstyle{definition}
\newtheorem{defn}[thm]{Definition}
\newtheorem{ass}[thm]{Assumptions}
\newtheorem{as}[thm]{Assumption}

\newtheorem{rem}[thm]{Remark}

\begin{document}

\frenchspacing

\title[Splitting with approximation]{Operator splittings and spatial approximations for evolution equations}

\author[A. B\'{a}tkai]{Andr\'{a}s B\'{a}tkai}
\address{ELTE TTK, Institute of Mathematics, 1117 Budapest, P\'{a}zm\'{a}ny P. s\'{e}t\'{a}ny 1/C, Hungary.}
\email{batka@cs.elte.hu}

\author[P. Csom\'{o}s]{Petra Csom\'{o}s}
\address{Technische Universit\"{a}t Darmstadt, Fachbereich Mathematik, Schlo{\ss}gartenstr. 7, 64289 Darmstadt, Germany}
\email{csomos@cs.elte.hu}

\author[G. Nickel]{Gregor Nickel}
\address{Universit\"{a}t Siegen, FB 6 Mathematik, Walter-Flex-Str. 3, 57068 Siegen, Germany.}
\email{grni@fa.uni-tuebingen.de}


\date\today
\dedicatory{To Rainer Nagel, our teacher and friend.}

\begin{abstract}
The convergence of various operator splitting procedures, such as the sequential, the Strang and the weighted splitting, is investigated in the presence of a spatial approximation. To this end the relevant notions and results of numerical analysis are presented, a variant of Chernoff's product formula is proved and the general Trotter-Kato approximation theorem is used. The methods are applied to  an abstract partial delay differential equation.
\end{abstract}
\maketitle


\section{Introduction}

Operator splitting procedures are usually used to solve partial differential equations numerically. They can be considered as certain time-discretization methods which simplify or even make the numerical treatment of differential equations possible.

\noindent The idea behind operator splitting procedures is the following. Usually, a certain physical phenomenon is the combined effect of several processes. The behaviour of a physical quantity is described by a partial differential equation in which the local time derivative depends on the sum of the sub-operators corresponding to the different processes. These sub-operators usually are of different nature. For each sub-problem corresponding to each sub-operator there might be an effective numerical method providing fast and accurate solutions. For the sum of these sub-operators, however, we usually cannot find an adequate method. Hence, application of operator splitting procedures means that instead of the sum we treat the sub-operators separately. The solution of the original problem is then obtained from the numerical solutions of the sub-problems. A simple splitting procedure was proposed by Bagrinovskii and Godunov (see \cite{Bagrinovskii-Godunov}) in 1957 as an example. However, they were systematically studied only in 1968 by Marchuk (see \cite{Marchuk}) and Strang (see \cite{Strang},\cite{Strang2}). Since then operator splitting procedures have been widely applied to various physical processes, see e.g., Zlatev \cite{Zlatev}, Zlatev and Dimov \cite{Zlatev-Dimov}, or Botchev et al. \cite{Botchev-Farago-Horvath}.

\noindent Although operator splitting combined with spatial approximations is widely used in applications, the theoretical convergence analysis of these problems seems to be still missing. One aim of this paper is to investigate systematically the convergence of various splitting procedures combined with spatial approximations. This important question is motivated by numerical analysis and real-life applications but in the present paper it will be investigated in the framework of functional analysis. This means that we consider and apply the well-known theorems of Chernoff and Trotter\,--\,Kato from another point of view. They are just mathematical tools to prove the convergence of numerical schemes solving partial differential equations.

\noindent First, we recall various notions from numerical analysis concentrating on finite difference methods and, in particular, the  sequential, Strang, and weighted splitting is introduced. In Section 3 we give an overview on the convergence of these splitting procedures, and the convergence is investigated in the presence of a spatial approximation. The last section is devoted to the application of these results to delay equations. Further applications along with numerical experiments and error estimates will be presented in a subsequent paper.

\noindent Since we apply the results of operator semigroup theory, our main reference regarding notations and terminology is Engel and Nagel \cite{Engel-Nagel}. To avoid technical complications, we restrict ourselves to the case of two sub-operators. The general problem can be treated analogously. Further, we make the following assumption for the rest of this work.
\begin{as}\label{gen_ass_abc}
Assume that $X$ is a Banach space, $A$ and $B$ are closed, densely defined linear operators generating the strongly continuous operator semigroups $T=\big(T(t)\big)_{t\ge 0}$ and $S=\big(S(t)\big)_{t\ge 0}$, respectively. Further, we assume that the closure of $A+B$, $\overline{A+B}$ with $D(A)\cap D(B)\subset D(\overline{A+B})$ is also the generator of a strongly continuous semigroup $U=\big(U(t)\big)_{t\ge 0}$.
\label{as:genass}
\end{as}


\section{Finite difference methods and splitting procedures}

In this section we recall some basic notions of numerical analysis needed for our later investigations. Though the results presented are standard and well-known, it seemed to be appropriate to present them to a reader having his interests in evolution equations and operator theory. Since we use the \emph{finite difference method} for solving the equations numerically, we restrict ourselves to some of its important properties. Our discussion follows Section 6 of  Atkinson and Han \cite{Atkinson-Han}, but another relevant reference is Richtmyer and Morton \cite{Richtmyer-Morton}. Then we introduce three types of \emph{splitting procedures} investigated in the present work: sequential, Stang, and  weighted splittings. We will see that splitting procedures are special finite difference methods.


\subsection{Finite difference methods}

\noindent Finite difference methods approximate the differential operator by a difference operator. Thus, instead of the differential equation, a system of algebraic equations has to be solved at each time step. A numerical scheme is only applicable if it can provide numerical solution arbitrarily close to the exact solution, i.e., if the method is \emph{convergent}. In what follows, we define the basic concepts of numerical analysis we need later on.

\noindent Let $X$ be a Banach space and $G:D(G)\subset X\to X$ be a linear (usually unbounded), closed, densely defined operator. Consider the abstract Cauchy problem, also called initial value problem
\begin{equation}
\left\{
\begin{aligned}
\frac{\mathrm{d}u(t)}{\mathrm{d}t}&=Gu(t), \qquad t\ge 0, \\
u(0)&=x\in X.
\end{aligned}
\right.
\label{acp_num}
\end{equation}
Recall that a function $u: \mathbb{R}^+\to X$ is called a (classical) solution of the abstract Cauchy problem (\ref{acp_num}) if  the function $u$ is continuously differentiable, $u(t)\in D(G)$, and
\begin{equation}
\lim\limits_{h\to 0}\left\|\frac{u(t+h)-u(t)}{h}-Gu(t)\right\|=0
\nonumber
\end{equation}
with $u(0)=x$. This limit is understood to be the right limit at $t=0$. Since $G$ generates the strongly continuous semigroup $\big(U(t)\big)_{t\ge 0}$, it is clear that problem (\ref{acp_num}) has a unique solution of the form
\begin{equation}
u(t)=U(t)x \qquad \mbox{for all } t\ge 0, \ x\in D(G).
\label{numanal_sg}
\end{equation}
\noindent Although the solutions of (\ref{acp_num}) are defined for all $t\ge 0$, with the help of a computer we can only hope for a solution in finite time. Therefore, from now on we choose an arbitrary but fixed ``end point'' $t_0\in\mathbb{R}^+$, that is, we solve (\ref{acp_num}) numerically on the time interval $[0,t_0]$.
\begin{defn}
A \emph{finite difference method} is a one-parameter family of linear operators
\begin{equation}
F(h): X\to X, \qquad \mbox{where } h\in(0,t_0], \ t_0\in\mathbb{R}^+ \mbox{ fixed}.
\nonumber
\end{equation}
The \emph{approximate solution} $u_m(t)$ at time $t=mh$ is then defined by
\begin{equation}
u_m(t)=F(h)^mx = F(t/m)^m x \qquad \mbox{for } m\in\mathbb{N}.
\nonumber
\end{equation}
We usually refer to $h$ as the \emph{time step} of the numerical method.
\label{def:diff_meth}
\end{defn}
\noindent Now we can define the consistency, the convergence, and the stability of a difference method.
\begin{defn}[Consistency]
The difference method is called \emph{consistent} if for all $x\in D(G)$ and for the corresponding solutions $u$ of the abstract Cauchy problem (\ref{acp_num}) the following holds:
\begin{equation}
\lim\limits_{h\to 0}\left\|\frac{F(h)u(t)-u(t+h)}{h}\right\|=0
\nonumber
\end{equation}
uniformly for $t\in[0,t_0]$.
\label{def:consistency}
\end{defn}
Roughly speaking, consistency means that the approximating difference equations converge in some sense to the original abstract Cauchy problem.

\begin{defn}[Convergence]
The difference method is called \emph{convergent at time $t\in[0,t_0]$} if for fixed $t$ and for any $x\in X$ we have
\begin{equation}
\lim\limits_{h_i\to 0}\left\|F(h_i)^{m_i}x-u(t)\right\|=0,
\label{numanal_conv}
\end{equation}
where $(m_i)_{i\in\mathbb{N}}$ is a sequence of integers and $(h_i)_{i\in\mathbb{N}}$ is a null-sequence of step sizes such that $m_ih_i=t$.
\label{def:numanal_conv}
\end{defn}
\begin{rem}
Convergence means that the numerical solution tends to the exact solution of the problem as the time step tends to zero. We  defined the convergence only at a fixed time level. However, from the theorems presented later it follows that the convergence is uniform for $t$ in compact intervals.
\end{rem}
\begin{defn}[Stability]
The difference method is called \emph{stable} if the family of operators
\begin{equation}
\left\{ F(h)^m: \ h\in(0,t_0], \ mh\le t_0, \ m\in\mathbb{N} \right\}
\nonumber
\end{equation}
is uniformly bounded, i.e., there exists a constant $M>0$ such that
\begin{equation}
\|F(h)^m\|\le M \qquad \mbox{for } mh\le t_0 \mbox{ and for all } h\in(0,t_0].
\nonumber
\end{equation}
\label{def:stability}
\end{defn}
\noindent In other words, stability means that the approximating solutions remain bounded.

\noindent The above three notions are connected through the following result (see  Lax \cite[Theorem. 8 in Section 34.3]{Lax} and  Atkinson and Han \cite[Theorem. 6.2.11]{Atkinson-Han}).
\begin{thm}[Lax Equivalence Theorem]
Let us assume that the initial value problem (\ref{acp_num}) is well-posed. Then a consistent finite difference scheme is convergent if and only if it is stable.
\label{thm:lax}
\end{thm}


\subsection{Splitting procedures}

In order to introduce the operator splitting procedures, we consider the following abstract Cauchy problem on the Banach space $X$ :
\begin{equation}
\left\{
\begin{aligned}
\frac{\mathrm{d}u(t)}{\mathrm{d}t}&=(A+B)u(t), \qquad t\ge 0, \\
u(0)&=x\in X,
\label{acpspl}
\end{aligned}
\right.
\tag{$\mathrm{ACP}$}
\end{equation}
where the operators $A$ and $B$ satisfy Assumption \ref{gen_ass_abc}.

\noindent
Since operator splitting procedures are time-discretization methods (see the explanation below), analogously to the numerical time step introduced in Definition \ref{def:diff_meth}, we choose the \emph{splitting time step} of length $\tau\in\mathbb{R}^+$.
\begin{defn}
The solution obtained by applying a splitting procedure is called \emph{split solution}. We remark that the split solution is only defined on the mesh
\begin{equation}
\omega_\tau:=\{k\tau,\ k\in\mathbb{N}\}.
\end{equation}
\end{defn}
\noindent In the following we collect three splitting procedures (see, e.g., Bagrinovskii and Godunov \cite{Bagrinovskii-Godunov}, Strang \cite{Strang}, Marchuk \cite{Marchuk}, Farag\'{o} \cite{Farago}), and consider the corresponding sub-problems with their solutions. Recall that we assumed that $\big(A,D(A)\big)$ and $\big(B,D(B)\big)$ generate the strongly continuous semigroups $\big(T(t)\big)_{t\ge 0}$ and $\big(S(t)\big)_{t\ge 0}$, respectively (cf. Assumption \ref{as:genass}).


\subsubsection{Sequential splitting.} As we have already mentioned in the Introduction, applying splitting procedures means that we solve sub-problems related to the sub-operators $A$ and $B$ separately. For sequential splitting this process can be formulated as follows.
\begin{equation}
\left\{
\begin{aligned}
\dfrac{\mathrm{d} u_1^{(k)}(t)}{\mathrm{d} t}&=Au_1^{(k)}(t), \qquad
t\in \big((k-1)\tau,k\tau\big], \smallskip \\
u_1^{(k)}((k-1)\tau)&=u^\mathrm{sq}((k-1)\tau), \smallskip \\
\end{aligned} \right.
\nonumber
\end{equation}
\begin{equation}
\left\{ \begin{aligned}
\dfrac{\mathrm{d} u_2^{(k)}(t)}{\mathrm{d} t}&=Bu_2^{(k)}(t), \qquad
t\in \big((k-1)\tau,k\tau\big], \smallskip \\
u_2^{(k)}((k-1)\tau)&=u_1^{(k)}(k\tau), \smallskip \\
u^\mathrm{sq}(k\tau)&:=u_2^{(k)}(k\tau),
\end{aligned} \right.
\nonumber
\end{equation}
with $k\in\mathbb{N}$ and $u^\mathrm{sq}(0)=x$. Using that $u_1^{(k)}(t)=T(t)u^\mathrm{sq}((k-1)\tau)$, and that $u_2^{(k)}(t)=S(t)u_1^{(k)}(k\tau)=S(t)T(\tau)u^\mathrm{sq}((k-1)\tau)$, we see by a simple induction argument that the  split solution $u^\mathrm{sq}(k\tau)$, obtained by applying the sequential splitting procedure, can be written as
\begin{equation}
u^\mathrm{sq}(k\tau)=[S(\tau)T(\tau)]^kx \qquad \mbox{for } k\in\mathbb{N},\, k\tau\leq t_0, \mbox{ and } x\in X.
\label{spl_sq0}
\end{equation}
\noindent Since this is a finite difference method, the convergence of a splitting procedure also needs to be investigated at a certain time level. Therefore, we write formula (\ref{spl_sq0}) in a more convenient way. Now we do not fix the value of the splitting time step $\tau$, but fix the value of $t\ge 0$. With $\tau:=t/n$ we obtain the  solution
\begin{equation}
u^\mathrm{sq}_n(t)=[S(t/n)T(t/n)]^nx
\label{spl_sq}
\end{equation}
for all $n\in\mathbb{N}$, $t\ge 0$ and $x\in X$. Hence, sequential splitting is a finite difference method with
\begin{equation*}
F^{sq}(h)=S(h)T(h), \qquad h\in(0,t_0].
\end{equation*}


\subsubsection{Strang splitting.} In the case of this splitting technique three sub-problems have to be solved for one splitting time step:
\begin{equation}
\left\{ \begin{aligned}
\dfrac{\mathrm{d} u_1^{(k)}(t)}{\mathrm{d} t}&=Au_1^{(k)}(t),
\qquad t\in \left((k-1)\tau,\left(k-\tfrac{1}{2}\right)\tau\right], \smallskip \\
u_1^{(k)}((k-1)\tau)&=u^\mathrm{St}((k-1)\tau), \smallskip \\
\end{aligned} \right.
\nonumber
\end{equation}
\begin{equation}
\left\{ \begin{aligned}
\dfrac{\mathrm{d} u_2^{(k)}(t)}{\mathrm{d} t}&=Bu_2^{(k)}(t),
\qquad t\in \big((k-1)\tau,k\tau\big], \smallskip \\
u_2^{(k)}((k-1)\tau)&=u_1^{(k)}\left(\left(k-\tfrac{1}{2}\right)\tau\right),
\end{aligned} \right.
\nonumber
\end{equation}
\begin{equation}
\left\{ \begin{aligned}
\dfrac{\mathrm{d} u_3^{(k)}(t)}{\mathrm{d} t}&=Au_3^{(k)}(t),
\qquad t\in \left(\left(k-\tfrac{1}{2}\right)\tau,k\tau\right], \smallskip \\
u_3^{(k)}\left(\left(k-\tfrac{1}{2}\right)\tau\right)&=u_2^{(k)}(k\tau), \smallskip \\
u^\mathrm{St}(k\tau)&:=u_3^{(k)}(k\tau),
\end{aligned} \right.
\nonumber
\end{equation}
where $u^\mathrm{St}(0)=x$ and $k\in\mathbb{N}$. The split solution can be written in this case, using the same arguments as above, as
\begin{equation}
u^\mathrm{St}(k\tau)=[T(\tau/2)S(\tau)T(\tau/2)]^kx \qquad \mbox{for } k\in\mathbb{N} \mbox { and } x\in X.
\label{spl_st0}
\end{equation}
As above, substituting $\tau:=t/n$ with $t\ge 0$ fixed, we have
\begin{equation}
u^\mathrm{St}_n(t)=[T(t/2n)S(t/n)T(t/2n)]^nx
\label{spl_st}
\end{equation}
for all $n\in\mathbb{N}$, $t\ge 0$ and $x\in X$. Hence, Strang splitting is a finite difference method with
\begin{equation*}
F^{St}(h)=T(h/2)S(h)T(h/2), \qquad h\in(0,t_0].
\end{equation*}


\subsubsection{Weighted splitting.} It is obtained by using two sequential splittings: first starting with operator $A$, and then starting with operator $B$. At time $t=k\tau$ the split solution is computed as a weighted average of the split solutions obtained by the two sequential splitting steps:
\begin{equation}
u^\mathrm{w}(k\tau)=\Theta u^{\mathrm{sq},AB}(k\tau)+(1-\Theta) u^{\mathrm{sq},BA}(k\tau),
\nonumber
\end{equation}
where $\Theta\in(0,1)$ is a given weight parameter, and $u^{\mathrm{sq},AB}(k\tau)$ and $u^{\mathrm{sq},BA}(k\tau)$ are the split solutions of the above two sequential splittings at time $k\tau$, respectively. In this case the split solution has the form
\begin{equation}
u^\mathrm{w}(k\tau)=\left[\Theta S(\tau)T(\tau) +(1-\Theta)T(\tau)S(\tau)\right]^k x
\label{spl_w0}
\end{equation}
for $k\in\mathbb{N}$ and $x\in X$. With varying splitting time step $\tau:=t/n$ and fixed $t\ge 0$, the above formula can be rewritten as
\begin{equation}
u^\mathrm{w}_n(t)=\left[\Theta S(t/n)T(t/n)+(1-\Theta)T(t/n)S(t/n)\right]^n x
\label{spl_w}
\end{equation}
for all $n\in\mathbb{N}$, $t\ge 0$, and $x\in X$. Hence, weighted splitting is a finite difference method with
\begin{equation*}
F^{w}(h)=\Theta S(h)T(h)+(1-\Theta)T(h)S(h), \qquad h\in(0,t_0].
\end{equation*}
\noindent  The case $\Theta =\frac{1}{2}$ is called \emph{symmetrically weighted splitting}.

\begin{rem}
Let us make some notes on the various splitting procedures. The consistency (and its order) of the different splittings were shown by Bj{\o}rhus in \cite{Bjorhus} and Farag\'{o} and Havasi in \cite{Farago-Havasi} for general $C_0$-semigroups under a restrictive domain condition on the generators. For matrices it is well-known that the sequential splitting is of first order, the Strang and the symmetrically weighted splitting is of second order. The weighted splittigs were introduced because they allow the use of parallel computations. For analytic and unitary (semi)groups in Hilbert spaces, the convergence of sequential and Strang splittings were also investigated see, e.g., Ichinose et al. \cite{Ichinose-etal}, Neidhardt and Zagrebnov \cite{Neidhardt-Zagrebnov} and \cite{Neidhardt-Zagrebnov2}, and Zagrebnov \cite{Zagrebnov}. In this context this means that the splitting procedures converge also in the operator norm and we get a uniform error estimate. It is a very interesting and promising direction, but unfortunately not applicable for a large class of equations, see the delay equation example in the last section. This is why we investigate strong convergence in this paper.
\end{rem}


\section{Convergence of the splitting procedures}

First we show that the investigated splitting procedures are convergent if the exact solutions of the sub-problems are known. Then the convergence is showed also for the case when the solutions of the sub-problems are only approximated by certain spatial and temporal numerical schemes.


\subsection{Convergence with exact solutions}

Since operator splitting procedures are finite difference methods, their convergence is defined as in Definition \ref{def:numanal_conv}. More precisely, a splitting procedure is called \emph{convergent at a fixed time level} $t\ge 0$ if
\begin{equation}
\lim\limits_{n\to\infty}\|[F^\mathrm{spl}(t/n)]^nx-U(t)x\|=0 \qquad \mbox{for all } x\in X,
\nonumber
\end{equation}
with index `sq', `St', or `w' for the sequential, Strang, and weighted splitting, respectively. The convergence of the splittings will be a consequence of Chernoff's Theorem (see Chernoff \cite{Chernoff}, and  Engel and Nagel \cite[Corollary III.5.3]{Engel-Nagel}). Let us observe that Chernoff's Theorem and one direction of Lax's Theorem \ref{thm:lax} state the same: the stable and consistent methods are convergent.

\noindent Since the stability condition appearing in Chernoff's Theorem will play an important role in obtaining the convergence of the splitting method, we cite Lemma 2.3 from the paper \cite{Csomos-Nickel}.
\begin{lem}
Let us assume that there exist constants $M\ge 1$ and $\omega\in\mathbb{R}$ such that
\begin{equation}
\|[S(t/n)T(t/n)]^n\|\le Me^{\omega t} \qquad \mbox{for all } t\ge 0,\ n\in\mathbb{N}.
\label{stab_0}
\end{equation}
Then the followings hold.
\renewcommand{\labelenumi}{(\roman{enumi})}
\begin{enumerate}
\item\label{stab_1}
There exist constants $M_1\ge 1$, $\omega_1\in\mathbb{R}$ such that
\begin{equation}
\|[S(t/n)T(t/n)]^{n-1}\|\le M_1e^{\omega_1 t} \qquad \mbox{for all } t\ge 0,\ n\in\mathbb{N}.
\nonumber
\end{equation}
\item\label{stab_2} There exist constants $M_2\ge 1$, $\omega_2\in\mathbb{R}$ such that
\begin{equation}
\|[T(t/n)S(t/n)]^n\|\le M_2e^{\omega_2 t} \qquad \mbox{for all } t\ge 0,\ n\in\mathbb{N}.
\nonumber
\end{equation}
\item\label{stab_3}
There exist constants $M_3\ge 1$, $\omega_3\in\mathbb{R}$ such that
\begin{equation}
\|[S(t/2n)T(t/n)S(t/2n)]^n\|\le M_3e^{\omega_3 t} \qquad \mbox{for all } t\ge 0,\ n\in\mathbb{N}.
\nonumber
\end{equation}
\end{enumerate}
\label{lem:stab}
\end{lem}
\noindent Lemma \ref{lem:stab} states that the stability of the sequential splitting implies the stability of the Strang and  of the reverse order splitting.

\begin{rem}
The stability assumption for the weighted splitting can be formulated as follows. There exist constants $M_4\ge 1$, $\omega_4\in\mathbb{R}$ such that
\begin{align}
\label{stab_weighted} \|[\Theta S(t/n)T(t/n)+(1-\Theta) T(t/n)S(t/n)]^n\|\le M_4e^{\omega_4 t} & \qquad \mbox{and} \\
\nonumber \|[\Theta T(t/n)S(t/n)+(1-\Theta) S(t/n)T(t/n)]^n\|\le M_4e^{\omega_4 t} &
\end{align}
for all  $t\ge 0$, $n\in\mathbb{N}$, where $\Theta\in[0,1]$. Notice that, in general Condition \eqref{stab_0} does not seem to imply the stability condition for the weighted splitting. The question of giving a simpler condition implying the stability condition for the weighted splitting is still an open problem.
\end{rem}

\noindent The following proposition is a consequence of Chernoff's Theorem.
\begin{cor}
Under the Assumption \ref{gen_ass_abc}, the sequential and the Strang splittings are convergent at a fixed time level $t>0$ if the stability condition (\ref{stab_0}) is satisfied. The weighted splitting is convergent if the stability condition \eqref{stab_weighted} is satisfied.
\label{prop:spl-conv-if-stab}
\end{cor}
\begin{proof}
\mbox{} \\
\noindent \emph{Sequential and Strang splittings:} See \cite[Corollaries 2.5 and 2.7]{Csomos-Nickel}. \\
\emph{Weighted splitting:} \\
\noindent In order to show the convergence, we apply Chernoff's Theorem to the operator
\begin{equation}
F^\mathrm{w}(h):=\Theta S(h)T(h)+(1-\Theta)T(h)S(h) \qquad \Theta\in(0,1).
\nonumber
\end{equation}
The stability holds due to our assumptions. Using the semigroup property, we obtain that
\begin{align}
\nonumber & \lim\limits_{h\to 0} \dfrac{[\Theta S(h)T(h)+(1-\Theta)T(h)S(h)]x-x}{h} \medskip \\
\nonumber = & \Theta \lim\limits_{h\to 0} \dfrac{S(h)T(h)x-x}{h}
+(1-\Theta)\lim\limits_{h\to 0} \dfrac{T(h)S(h)x-x}{h} \medskip \\
\nonumber =& \Theta(Ax+Bx)+(1-\Theta)(Ax+Bx)=(A+B)x
\end{align}
for all $x\in D(A)\cap D(B)$ and $\Theta\in(0,1)$, since the topology of pointwise convergence on a dense subset of $X$ and the topology of uniform convergence on relatively compact subsets of $X$ coincide on bounded subsets of $\mathcal{L}(X)$, see e.g., Engel and Nagel \cite[Proposition A.3]{Engel-Nagel}. Thus, the consistency criterion of Chernoff's Theorem holds as well.
\end{proof}

\begin{rem}\label{rem:stab_cond}
Notice that in the important special cases, where the semigroups $S$ and $T$ are quasi-contractions or commuting, all the stability conditions are automatically satisfied.
\end{rem}


\subsection{Convergence with approximations}
\label{sc:approx}

In the previous section we investigated the convergence of the splitting procedures in the case when the sub-problems are solved exactly. In concrete problems, however, the exact solutions are not known. Therefore the use of a certain approximation scheme is needed to solve the sub-problems. When a partial differential equation is to be solved applying a splitting procedure together with approximation schemes, we have to follow these steps.
\renewcommand{\labelenumi}{\textbf{\arabic{enumi}.}}
\begin{enumerate}
\item The spatial differential operator is split into sub-operators of simpler form. (For instance according to the different physical phenomena or space directions, etc.)
\item Each sub-operator is approximated by an appropriate spatial discretization scheme (called \emph{semi-disc\-re\-ti\-za\-tion}). Then we obtain systems of ordinary differential equations corresponding to the sub-operators.
\item Each solution of the semi-discretized system is obtained by using a time-dis\-cre\-ti\-za\-tion method.
\end{enumerate}
\noindent In this section we investigate the case when the semi-discretized systems can be solved analytically, i.e., without using a time-disc\-re\-ti\-za\-tion method. That is, we assume that the semigroups are approximated by approximate semigroups (step \textbf{1} and step \textbf{2}). In the second case the semigroups are approximated by operators which are not necessarily semigroups. They represent the case when the solutions of the semi-discretized systems are obtained by using a time-disc\-re\-ti\-za\-tion scheme (step \textbf{1}, step \textbf{2}, and step \textbf{3}). This case will be investigated in our forthcoming work. We remark that the convergence of the splitting together with the time-disc\-re\-ti\-za\-tion method (without the spatial approximation scheme: step \textbf{1} and step \textbf{3}) is investigated, e.g., in \cite{Csomos-Farago}. \\

\noindent Consider the abstract Cauchy problem \eqref{acpspl} on the Banach space $X$ for the sum of the generators $\big(A,D(A)\big)$ and $\big(B,D(B)\big)$. As introduced e.g. by Ito and Kappel in \cite[Section 4.1]{Ito-Kappel} and Pazy in \cite[Section 3.6]{Pazy}, we define approximate spaces (``mesh'') and projection-like operators between the approximate spaces and the original space $X$.
\begin{defn}
Let $X_m$, $m\in\mathbb{N}$ be Banach spaces and take operators
\begin{equation}
P_m:\ X\rightarrow X_m \qquad \mbox{and} \qquad J_m:\ X_m\rightarrow X
\nonumber
\end{equation}
fulfilling  the following properties:
\renewcommand{\labelenumi}{(\roman{enumi})}
\begin{enumerate}
\item $P_mJ_m=I_m$ for all $m\in\mathbb{N}$, where $I_m$ is the identity operator in $X_m$,
\item $\lim\limits_{m\rightarrow\infty}J_mP_mx=x$ for all $x\in X$,
\item $\|J_m\|\le M_J$ and $\|P_m\|\le M_P$ for all $m\in\mathbb{N}$ and some given constants $M_J,M_P>0$.
\end{enumerate}
\label{def:proj}
\end{defn}

\noindent The operators $P_m$ together with the spaces $X_m$ usually refer to a kind of spatial discretization (triangulation, Galerkin approximation, Fourier coefficients, etc.), the spaces $X_m$ are usually finite dimensional spaces and the operators $J_m$ refer to the interpolation method, how we associate specific elements of the function space to the elements of the approximating spaces (linear/polynomial/spline interpolation, etc.).

\noindent First we split the operator $A+B$ appearing in the original problem \eqref{acpspl} into the sub-operators $A$ and $B$. In order to obtain the semi-discretized systems, the sub-operators $A$ and $B$ in equation \eqref{acpspl} have to be approximated by  operators $A_m$ and $B_m$ for $m\in\mathbb{N}$ fixed. Assume that the operators $A_m$ and $B_m$ generate the strongly continuous semigroups $\big(T_m(t)\big)_{t\ge 0}$ and $\big(S_m(t)\big)_{t\ge 0}$ on the space $X_m$, respectively. For the analysis of the convergence, we need the following definitions.

\begin{defn}
We consider the following properties of the semigroups $\big(T_m(t)\big)_{t\ge 0}$, $\big(S_m(t)\big)_{t\ge 0}$, $m\in\mathbb{N}$, and their generators $\big(A_m,D(A_m)\big)$, $\big(B_m,D(B_m)\big)$, $m\in\mathbb{N}$, respectively.
\renewcommand{\labelenumi}{(\roman{enumi})}
\renewcommand{\labelenumii}{(\alph{enumii})}
\begin{enumerate}
\item \emph{Stability:} \\
there exist constants $M,M_T,M_S\ge 1$ and $\omega, \omega_T,\omega_S\in\mathbb{R}$ such that
\begin{enumerate}
\item $\|T(h)\|\le M_T\mathrm{e}^{\omega_T h}$ \quad and \quad $\|T_m(h)\|\le M_T\mathrm{e}^{\omega_T h}$,
\item $\|S(h)\|\le M_S\mathrm{e}^{\omega_S h}$ \quad and \quad $\|S_m(h)\|\le M_S\mathrm{e}^{\omega_S h}$,
\end{enumerate} \medskip for all $h>0$, and either \medskip
\begin{equation}
\|[ S_m(t/n)T_m(t/n)]^k\|\le Me^{k\omega\frac{t}{n}} \qquad \mbox{for all } t\ge 0,\ k,n,m\in\mathbb{N}
\label{stab_approx}
\end{equation}
in case of the sequential and Strang splittings, or
\begin{equation}
\| \left[\Theta  S_m(t/n)T_m(t/n)+ (1-\Theta) T_m(t/n)S_m(t/n)\right]^k\|\le Me^{k\omega\frac{t}{n}}
\label{stab_approx_w}
\end{equation}
for a $\Theta\in[0,1]$ and for all $t\ge 0,\ k,n,m\in\mathbb{N}$ in case of the weighted splitting.
\item \emph{Consistency:}
\begin{enumerate}
\item $\lim\limits_{m\to\infty}J_mA_mP_mx=Ax$ \qquad for all $x\in D(A)$,
\item $\lim\limits_{m\to\infty}J_mB_mP_mx=Bx$ \qquad for all $x\in D(B)$.
\end{enumerate}
\end{enumerate}
The semigroups $T_m$, $S_m$, $m\in\mathbb{N}$, are called \emph{approximate semigroups}, and their generators $\big(A_m,D(A_m)\big)$, $\big(B_m,D(B_m)\big)$ are called \emph{approximate generators} if they possess the above properties.
\label{def:stab_conv}
\end{defn}
\begin{cor}
From the assumptions in Definition \ref{def:stab_conv} and from the Trotter\,--\,Kato Approximation Theorem (see Ito and Kappel \cite[Theorem 2.1]{Ito-Kappel2}) it follows that the approximate semigroups converge to the original semigroups uniformly, that is: \\
\noindent Convergence:
\renewcommand{\labelenumi}{(\alph{enumi})}
\begin{enumerate}
\item $\lim\limits_{m\to\infty}J_mT_m(h)P_mx=T(h)x$ \qquad $\forall x\in X$ and uniformly for $h\in[0,t_0]$,
\item $\lim\limits_{m\to\infty}J_mS_m(h)P_mx=S(h)x$ \qquad $\forall x\in X$ and uniformly for $h\in[0,t_0]$.
\end{enumerate}
\label{cor:conv}
\end{cor}
\noindent We remark that the stability condition (i)(a) or (b) in Definition \ref{def:stab_conv} is fulfilled for many cases in applications, for example, for the semidiscrete Galerkin method for parabolic equations (see Larsson and Thom\'{e}e \cite[Section 10.1] {Larsson-Thomee}) or a large class of Galerkin type approximations for hyperbolic or more general equations (as an example, see Fabiano \cite{Fabiano} or Fabiano and Turi \cite{Fabiano-Turi}).
\begin{defn}
For the case of spatial approximation, we define the split solutions of \eqref{acpspl} as
\begin{equation}
\nonumber u^\mathrm{spl}_{n,m}(t) := J_m[F^\mathrm{spl}_{m}(t/n)]^nP_mx
\end{equation}
for $m,n\in\mathbb{N}$ fixed and for $x\in X$, where index `spl' is `sq', `St', or `w'. The operators $F_m$, describing the approximation schemes together with the splitting procedures have the form
\begin{align}
\label{sq-approx} F^\mathrm{sq}_{m}(h) &:= S_m(h)T_m(h), \\
\label{st-approx} F^\mathrm{St}_{m}(h) &:= T_m(h/2)S_m(h)T_m(h/2), \\
\label{w-approx} F^\mathrm{w}_{m}(h) &:=
\Theta S_m(h)T_m(h)+(1-\Theta)T_m(h)S_m(h), \qquad \Theta\in(0,1)
\end{align}
for the sequential, Strang, and weighted splittings, respectively, with $h\in(0,t_0]$.
\label{def:split_solutions_approx}
\end{defn}
\begin{defn}
The numerical method for solving problem \eqref{acpspl} described above is \emph{convergent at a fixed time level $t>0$} if for all $\varepsilon>0$ there exists $N\in\mathbb{N}$ such that for all $n,m\ge N$ we have
\begin{equation}
\left\|u^\mathrm{spl}_{n,m}(t)-u(t)\right\|\le\varepsilon,
\nonumber
\end{equation}
where the index `spl' refers to `sq', `St', or `w', respectively. This is the usual well-known notion of the convergence of a sequence with two indices and we will use the notation
\begin{equation*}
\lim_{n,m\to\infty} u^\mathrm{spl}_{n,m}(t) = u(t)
\end{equation*}
to express this.
\label{def:conv-approx}
\end{defn}
\begin{rem}
Observe that, analogously to the case of exact solutions (i.e., splitting without approximation), the stability condition \eqref{stab_approx} implies the stability of the reversed order and the Strang splitting. Since the proof is analogous to the one of Lemma \ref{lem:stab}, we omit it.
\end{rem}
\noindent This means that the sequential and Strang splittings fulfill their stability condition if the stability condition (\ref{stab_approx}) of the sequential splitting (with approximations) holds. Therefore, in this case, it suffices to control only this condition. Notice, however, that for the weighted splitting we need the more complicated condition \eqref{stab_approx_w} (with approximation).

\noindent Notice that the convergence of a sequence of two indices is in general difficult to treat, except in a well-known special case we recall here in a form we shall need in our proofs later on.
\begin{lem}
Let $X$ be a Banach space, $D\subset X$ a dense subset in $X$, and $V(h),\, V_m(h): D\to X$ operator for $h\in [0,t_0]$ and $\overline V:D\to X$. Let us assume that
\begin{align}
\label{lim1} & \exists \lim\limits_{m\to\infty} V_m(h)x=V(h)x \quad \mbox{for all } x\in D, \mbox{ uniformly for } h\in[0,t_0] \mbox{ and} \\
\label{lim2} & \exists \lim\limits_{h\to 0} V(h)x=\overline Vx \quad \mbox{for all } x\in D.
\end{align}
Then
\begin{equation}
\lim\limits_{n,m\to\infty} V_m(t/n)x=\overline Vx \qquad \mbox{holds for all } x\in D
\nonumber
\end{equation}
and for all $t\in[0,t_0]$ fixed.
\label{lem:limits}
\end{lem}

\noindent In order to prove the convergence of operator splitting in this case, we state a modified version of Chernoff's Theorem, which is applicable for approximate semigroups as well. The version we present here is a slight modification of Pazy \cite[Theorem. 6.7]{Pazy}, which is modified to fit our setting.
\begin{thm}[Modified Chernoff's Theorem]
Consider  a sequence of functions $F_m:\mathbb{R}^+\rightarrow\mathcal{L}(X_m)$, $m\in\mathbb{N}$, satisfying
\begin{equation}
F_m(0)=I_m \qquad \mbox{for all } m\in\mathbb{N},
\label{chernoff2-1}
\end{equation}
and that there exists a constant $M\ge 1$ such that
\begin{equation}
\|[F_m(t)]^k\|_{\mathcal{L}(X_m)}\le M \qquad \mbox{for all } t\ge 0,\ m,k\in\mathbb{N}.
\label{chernoff2-2}
\end{equation}
Assume further that
\begin{equation*}
\exists \lim\limits_{m\rightarrow\infty}\dfrac{J_mF_m(t/n)P_mx-J_m P_m x}{t/n}
\end{equation*}
uniformly in $n\in\mathbb{N}$ and that
\begin{equation}
Gx:= \lim\limits_{n\rightarrow\infty}\lim\limits_{m\rightarrow\infty}\dfrac{J_mF_m(t/n)P_mx-J_m P_m x}{t/n}
\label{chernoff2-3}
\end{equation}
exists for all $x\in D\subset X$ and an arbitrary $t>0$, where $D$ and $(\lambda_0-G)D$ are dense subspaces in $X$ for $\lambda_0>0$. Then the closure $\overline G$ of $G$ generates a bounded strongly continuous semigroup $\left( U(t) \right)_{t\ge 0}$, which is given by
\begin{equation}
U(t)x=\lim\limits_{n,m\rightarrow\infty} J_m[F_m(t/n)]^n P_m x
\label{convergence2}
\end{equation}
for all $x\in X$ uniformly for $t$ in compact intervals.
\label{thm:chernoff2}
\end{thm}
\begin{proof}
For $h>0$ define
\begin{equation}
G_{h,n,m}:=\dfrac{F_m(h/n)-I_m}{h/n}\in\mathcal{L}(X_m)
\nonumber
\end{equation}
for all $n,m\in\mathbb{N}$. Observe that $\lim\limits_{n,m\to\infty}J_mG_{h,n,m}P_mx=Gx$ for all $x\in D$. Every semigroup $(e^{tG_{h,n,m}})_{t\ge 0}$ satisfies
\begin{equation}
\left\|e^{tG_{h,n,m}}\right\| \le e^{-tn/h}\left\|e^{tnF_m(h/n)/h}\right\| \le e^{-tn/h}\sum\limits_{k=0}^\infty\dfrac{t^kn^k}{h^kk!}\|[F_m(h/n)]^k\| \le M
\label{mod-chernoff-M}
\end{equation}
for every $t\ge 0$. This shows that the assumptions of the Trotter\,--\,Kato Theorem (see Ito and Kappel \cite[Theorem. 4.2]{Ito-Kappel} together with Engel and Nagel \cite[Theorem III.4.8]{Engel-Nagel}) are fulfilled, and we can apply it first taking limit in $m\to\infty$ (which is uniform in $n\in\mathbb{N}$), and then taking limit as $n\to\infty$. Hence, the closure $\overline G$ of $G$ generates a strongly continuous semigroup $U$ given by
\begin{equation}
\lim\limits_{n,m\to\infty}\|U(t)x-J_me^{tG_{t,n,m}}P_mx\|=0 \qquad \mbox{for all } x\in X \label{tri1}
\end{equation}
uniformly for $t$ in compact intervals.
On the other hand, we have by Lemma III.5.1. in Engel and Nagel \cite{Engel-Nagel}:
\begin{align}
\nonumber & \left\|J_m e^{tG_{t,n.m}}P_m x-J_m[F_m(t/n)]^n P_m x\right\| \\
\label{tri2} =& \left\|J_me^{n(F_m(t/n)-I_m)}P_mx-J_m[F_m(t/n)]^n P_m x\right\| \\
\nonumber \le & \sqrt{n}M\|J_mF_m(t/n)P_m x-J_mP_m x\| \\
\nonumber =& \dfrac{tM}{\sqrt{n}}\left\|\dfrac{J_mF_m(t/n)P_mx-J_m P_m x}{t/n}\right\| \xrightarrow{n,m\to\infty} 0
\end{align}
for all $x\in D$, and uniformly for $t$ in compact intervals. The combination of (\ref{tri1}) and (\ref{tri2}) yields
\begin{align}
\nonumber & \|U(t)x-J_m[F_m(t/n)]^n P_mx\| \\
\nonumber \le & \|U(t)x-J_me^{tG_{t,n,m}}P_mx\|+\|J_me^{tG_{t,n,m}}P_mx-J_m[F_m(t/n)]^n P_mx\| \xrightarrow{n,m\to\infty} 0
\end{align}
for all $x\in D$, and uniformly for $t$ in compact intervals. By the uniform boundedness principle the statement follows for all $x\in X$.
\end{proof}
\begin{rem}\label{cor:chernoff2v}
Theorem \ref{thm:chernoff2} remains valid  in the case when the stability conditions reads as
\begin{equation}
\|[F_m(t)]^k\|\le Me^{k\omega{t}}
\label{chernoff2-2v}
\end{equation}
for all $t\ge 0$ and $m\in\mathbb{N}$, $k\in\mathbb{N}$,
and for some constants $M\ge 1$, $\omega\in\mathbb{R}$, as we can see it from a standard rescaling procedure, see Engel and Nagel \cite[Corollary III.5.3]{Engel-Nagel}.
\end{rem}


\noindent In the remainder of this subsection, we consider the convergence of the various splitting procedures.
\begin{lem}
Let $J_m,P_m,T_m$ be operators defined in Definitions \ref{def:proj} and \ref{def:stab_conv}. Then the following holds for all $t\in [0,t_0]$:
\begin{equation}
\lim\limits_{n,m\to\infty}\dfrac{J_mT_m(t/n)P_m x-J_mP_m x}{t/n}=Ax \qquad \mbox{for all } x\in D(A),
\nonumber
\end{equation}
where the limit as $m\to\infty$ is uniform in $n\in\mathbb{N}$.
\label{lem:T_m}
\end{lem}
\begin{proof}
It suffices to prove that the operators
\begin{equation}
V_m(h)=\dfrac{J_mT_m(h)P_m-J_mP_m}{h}
\nonumber
\end{equation}
fulfill the conditions in Lemma \ref{lem:limits}. In order to determine the first limit \eqref{lim1}, the following norm has to be investigated for all $x\in D(A)$:
\begin{align}
\nonumber & \left\| \dfrac{J_mT_m(h)P_mx-J_mP_mx}{h}-\dfrac{T(h)x-x}{h} \right\|
= \dfrac{1}{h}\left\| \int\limits_0^h J_mA_mT_m(s)P_mx\mathrm{d}s - \int\limits_0^h AT(s)x\mathrm{d}s \right\| \\
\nonumber \le & \sup\limits_{s\in[0,t_0]} \left\| J_mA_mT_m(s)P_mx-AT(s)x \right\|
= \sup\limits_{s\in[0,t_0]} \left\| J_mT_m(s)A_mP_mx-T(s)Ax \right\| \\
\nonumber \le & \sup\limits_{s\in[0,t_0]} \left\| J_mT_m(s)P_m[J_mA_mP_mx-Ax]+[J_mT_m(s)P_m-T(s)]Ax \right\| \\
\nonumber \le & \sup\limits_{s\in[0,t_0]} \|J_m\|\|T_m(s)\|\|P_m\|\|J_mA_mP_mx-Ax\|
+\sup\limits_{s\in[0,t_0]} \|[J_mT_m(s)P_m-T(s)]Ax\|.
\end{align}
According to Definition \ref{def:stab_conv}, the term $\|J_mA_mP_mx-Ax\|$ tends to 0 as $m$ tends to infinity. Since $y:=Ax$ is a fixed element in the Banach space $X$, $\|J_mT_m(s)P_my-T(s)y\|$ tends to 0 uniformly in $h$ as $m$ tends to infinity because of Corollary \ref{cor:conv}. Operators $J_m$ and $P_m$ were assumed to be bounded. Term $\|T_m(s)\|$ was also assumed to be exponentially bounded with a constant $\omega_T$, which is independent of $m$. Therefore,
\begin{equation}
\sup\limits_{s\in[0,t_0]}\|T_m(s)\|\le \sup\limits_{s\in[0,t_0]} M_T\mathrm{e}^{\omega_T s}
\le M_T\mathrm{e}^{\max\{0,\omega_T\} t_0} = const. < \infty.
\nonumber
\end{equation}
Hence, the above difference tends to 0 uniformly in $h$, therefore, \eqref{lim1} holds. The second limit \eqref{lim2} can be obtained by using the definition of the generator:
\begin{equation}
\lim\limits_{h\to 0}\dfrac{T(h)x-x}{h}=Ax \qquad \forall x\in D(A).
\nonumber
\end{equation}
Thus, the statement is proved.
\end{proof}
\begin{cor}
The following statement can be proved similarly as Lemma \ref{lem:T_m}:
\begin{equation}
\lim\limits_{n,m\to\infty}\dfrac{J_mS_m(t/n)P_mx-J_mP_mx}{t/n}=Bx \qquad \forall x\in D(B).
\nonumber
\end{equation}
\label{cor:S_m}
\end{cor}
\noindent Now we really turn our attention to the convergence of the different splitting procedures. First we show the convergence of the split solution defined in (\ref{sq-approx}), i.e., in the case when the sequential splitting is applied.
\begin{thm}
The sequential splitting (\ref{sq-approx}) is convergent at  time level $t>0$ if the stability condition holds for the approximate semigroups, and the approximate generators are consistent according to Definition \ref{def:stab_conv}.
\label{thm:sq-conv-approx}
\end{thm}
\begin{proof}
According to Chernoff's Theorem \ref{thm:chernoff2} the sequential splitting is convergent if the stability \eqref{chernoff2-2v} and the consistency \eqref{chernoff2-3} hold for the operator
\begin{equation}
F_m(h):=S_m(h)T_m(h).
\end{equation}
The stability condition \eqref{chernoff2-2v} is fulfilled, since we assumed that \eqref{stab_approx} holds. In order to prove the consistency criterion \eqref{chernoff2-3}, the following equation should be verified:
\begin{equation}
\lim\limits_{n,m\to\infty}\dfrac{J_mF_m(t/n)P_mx-J_mP_mx}{t/n}=(A+B)x \qquad \mbox{for all } x\in D(A)\cap D(B).
\nonumber
\end{equation}
We investigate the following limit:
\begin{align}
\nonumber & \lim\limits_{n,m\to\infty}\dfrac{J_mS_m(t/n)T_m(t/n)P_mx-J_mP_mx}{t/n} \\
\nonumber =& \lim\limits_{n,m\to\infty}J_mS_m(t/n)P_m\dfrac{J_mT_m(t/n)P_mx-J_mP_mx}{t/n}\\
\nonumber +& \lim\limits_{n,m\to\infty}\dfrac{J_mS_m(t/n)P_mx-J_mP_mx}{t/n},
\end{align}
where we apply Lemma \ref{lem:limits} for each term. Choosing
\begin{equation}
V_m(h):=J_mS_m(h)P_m
\nonumber
\end{equation}
and applying Corollary \ref{cor:conv}, we obtain that
\begin{align}
\nonumber & \lim\limits_{m\to\infty}J_mS_m(h)P_mx=S(h)x \qquad \forall x\in X \mbox{ uniformly for } h\in[0,t_0] \mbox{ and} \\
\nonumber & \lim\limits_{h\to 0}S(h)x=x \qquad \mbox{for all } x\in X.
\end{align}
Notice further that the set $\left\{\frac{1}{h}(J_mT_m(h)P_mx-J_mP_mx):\,h\in[0,t_0]\right\}$ is relative compact for $x\in D(A)$, and that on compact sets the strong and the uniform convergence is equivalent.
Hence,  Lemma \ref{lem:T_m} and Corollary \ref{cor:S_m} can be applied, thus, we obtain that
\begin{align}
& \lim\limits_{n,m\to\infty}\dfrac{J_mF_m(t/n)P_mx-J_mP_mx}{t/n}\\
\nonumber &=\lim\limits_{n,m\to\infty}\dfrac{J_mS_m(t/n)T_m(t/n)P_mx-J_mP_mx}{t/n}=(IA+B)x
\nonumber
\end{align}
for all $x\in D(A)\cap D(B)$. Here $I\in\mathcal{B}(X)$ denotes the identity operator. This means that the consistency criterion of Chernoff's Theorem is fulfilled for the sequential splitting, thus, it is convergent with spatial approximations.
\end{proof}


\noindent Now we prove the convergence of the split solution defined in \eqref{st-approx}, i.e., in the case when the Strang splitting is applied.
\begin{thm}
The Strang splitting \eqref{st-approx} is convergent at time level $t>0$ if the stability condition holds for the approximate semigroups, and the approximate generators are consistent according to Definition \ref{def:stab_conv}.
\label{thm:st-conv-approx}
\end{thm}
\begin{proof}
As shown in Lemma \ref{lem:stab} the stability condition for Strang splitting follows from that of sequential splitting which has already been assumed in Definition \ref{def:stab_conv}. Hence, now it suffices to check whether the consistency \eqref{chernoff2-3} holds for the operator
\begin{equation}
F_m(h):=T_m(h/2)S_m(h)T_m(h/2).
\nonumber
\end{equation}
Thus, we should investigate the following limit
\begin{align}
\nonumber & \lim\limits_{n,m\to\infty}\dfrac{J_mF_m(t/n)P_mx-J_mP_mx}{t/n}\\
\nonumber =& \lim\limits_{n,m\to\infty}\dfrac{J_mT_m(t/2n)S_m(t/n)T_m(t/2n)P_mx-J_mP_m x}{t/n} \\
\nonumber =& \lim\limits_{n,m\to\infty}J_mT_m(t/2n)S_m(t/n)P_m\dfrac{J_mT_m(t/2)P_mx-J_mP_mx}{t/n} \\
\nonumber +& \lim\limits_{n,m\to\infty}J_mT_m(t/2n)P_m\dfrac{J_mS_m(t/2)P_mx-J_mP_mx}{t/n}\\
\nonumber +& \lim\limits_{n,m\to\infty}J_mP_m\dfrac{J_mT_m(t/2)P_mx-J_mP_mx}{t/n}.
\end{align}
We can apply Lemma \ref{lem:T_m} and Corollary \ref{cor:S_m} again to each term, and obtain that
\begin{align}
\nonumber & \lim\limits_{n,m\to\infty}\dfrac{J_mF_m(t/n)P_mx-J_mP_mx}{t/n}\\
\nonumber =& \lim\limits_{n,m\to\infty}\dfrac{J_mT_m(t/2n)S_m(t/n)T_m(t/2n)P_mx-J_mP_mx}{t/n}  \\
\nonumber =& I\cdot I\cdot\tfrac{1}{2}Ax+I\cdot Bx+\tfrac{1}{2}Ax =(A+B)x%
\end{align}
for all $x\in D(A)\cap D(B)$. Therefore, the Strang splitting is convergent with spatial approximations.
\end{proof}


\noindent Finally, let us prove the convergence of the split solution defined in (\ref{w-approx}), i.e., in the case when the weighted splitting is applied.
\begin{thm}
The weighted splitting (\ref{w-approx}) is convergent at time level $t>0$ if the stability condition \eqref{stab_approx_w} holds for the approximate semigroups, and the approximate generators are consistent according to Definition \ref{def:stab_conv}.
\label{thm:w-conv-approx}
\end{thm}
\begin{proof}
Similar to the calculations presented in proofs of Theorems \ref{thm:sq-conv-approx} and \ref{thm:st-conv-approx} for the operator
\begin{equation}
F_m(h):=\Theta S_m(h)T_m(h)+(1-\Theta)T_m(h)S_m(h), \ \Theta\in(0,1).
\nonumber
\end{equation}
\end{proof}

\noindent Summarizing, we can say that by Chernoff's Theorem \ref{thm:chernoff2} the stability condition (\ref{stab_approx}) or \eqref{stab_approx_w} implies the convergence of the splitting procedures in the case of consistent approximations.


\section{Operator splittings for delay equations}

Since many physical processes depend on a former state of the system as well, they have to be described by \emph{delay differential equations} containing a term  depending on the \emph{history function} (see \cite{Batkai-Piazzera}, Kappel \cite{Kappel}). Although these differential equations cannot be written as an abstract Cauchy problem on the original state space $X$, their solutions can be obtained by an operator semigroup on an appropriate function space (called \emph{phase space}). For a systematic treatment of the problem we refer to the monograph \cite{Batkai-Piazzera}. We concentrate here on the case of bounded delay operators, while unbounded delay operators will be treated in a subsequent work.

\noindent Consider the \emph{abstract delay equation} in the following form (see, e.g., \cite{Batkai-Piazzera}):
\begin{equation}
\left\{
\begin{aligned}
\frac{\mathrm{d}u(t)}{\mathrm{d}t}&=Cu(t)+\Phi u_t, \qquad t\ge 0, \\
u(0)&=x\in X, \\
u_0&=f\in\mathrm{L}^1\big([-1,0],X\big)
\end{aligned}
\right.
\tag{DE}
\label{delay}
\end{equation}
on the Banach space $X$, where $\big(C,D(C)\big)$ is a generator of a strongly continuous semigroup on $X$, and $\Phi:\mathrm{L}^{1}\big([-1,0],X\big) \to X$ is a bounded and linear operator. The \emph{history function} $u_t$ is defined by $u_t(\sigma):=u(t+\sigma)$ for $\sigma\in[-1,0]$. \\

\noindent In order to rewrite (\ref{delay}) as an abstract Cauchy problem, we take the product space $\mathcal{E}:=X\times \mathrm{L}^1\big([-1,0],X\big)$ and the new unknown function as
\begin{equation}
t\mapsto\mathcal{U}(t):=\binom{u(t)}{u_t}\in\mathcal{E}.
\nonumber
\end{equation}
Then (\ref{delay}) can be written as an abstract Cauchy problem on the space $\mathcal{E}$ in the following way:
\begin{equation}
\left\{
\begin{aligned}
\frac{\mathrm{d}\mathcal{U}(t)}{\mathrm{d}t}&=\mathcal{G}\mathcal{U}(t), \qquad t\ge 0, \\
\mathcal{U}(0)&=\tbinom{x}{f}\in\mathcal{E},
\end{aligned}
\right.
\tag{$\mathcal{ACP}$}
\label{acp_delay}
\end{equation}
where the operator $\big(\mathcal{G},D(\mathcal{G})\big)$ is given by the matrix
\begin{equation}
\mathcal{G}:=\left(\begin{array}{cc} C & \Phi \\ 0 & \frac{d}{d\sigma} \end{array}\right)
\end{equation}
on the domain
\begin{equation}
D(\mathcal{G}):=\left\{\tbinom{\xi}{\eta}\in D(C)\times \mathrm{W}^{1,1}\big([-1,0],X\big): \ \eta(0)=\xi  \right\}.
\nonumber
\end{equation}
\noindent It is shown in \cite[Corollary 3.5, Proposition 3.9]{Batkai-Piazzera} that the delay equation (\ref{delay}) and the abstract Cauchy problem (\ref{acp_delay}) are equivalent, i.e., they have the same solutions. More precisely, the first coordinate of the solution of (\ref{acp_delay}) always solves (\ref{delay}). Due to this equivalence, the delay equation is well-posed if and only if the operator $\big(\mathcal{G},D(\mathcal{G})\big)$ generates a strongly continuous semigroup on the space $\mathcal{E}$. \\

\noindent The following case was partly investigated in the paper \cite{Csomos-Nickel}.
\begin{ass}
\mbox{}
\renewcommand{\labelenumi}{(\alph{enumi})}
\begin{enumerate}
\item The operator $\big(C,D(C)\big)$ generates a strongly continuous contraction semigroup $\big(V(t)\big)_{t\ge 0}$ on $X$.
\item The delay operator $\Phi:\ \mathrm{L}^1\big([-1,0],X\big) \to X$ is bounded.
\end{enumerate}
\label{ass1}
\end{ass}
\noindent Since the delay operator $\Phi$ is bounded, the delay equation (\ref{delay}) is well-posed by \cite[Theorem. 3.26]{Batkai-Piazzera}. In order to apply an operator splitting procedure, we split the operator in (\ref{acp_delay}) as
\begin{equation}
\mathcal{G}=\mathcal{A}+\mathcal{B},
\nonumber
\end{equation}
where the sub-operators have the forms
\begin{equation}
\begin{array}{ll}
\mathcal{A}:=\left(\begin{array}{cc} C & 0 \\ 0 & \frac{\mathrm{d}}{\mathrm{d}\sigma} \end{array}\right), & \quad D(\mathcal{A}):=D(\mathcal{G}), \medskip \\
\mathcal{B}:=\left(\begin{array}{cc} 0 & \Phi \\ 0 & 0 \end{array}\right), & \quad D(\mathcal{B}):=\mathcal{E}.
\end{array}
\label{matrices-delay}
\end{equation}
Since $C$ is a generator and $\Phi$ is bounded, the operators $\mathcal{A}$ and $\mathcal{B}$ generate the strongly continuous semigroups $\big(\mathcal{T}(t)\big)_{t\ge 0}$ and $\big(\mathcal{S}(t)\big)_{t\ge 0}$, respectively. It is shown in \cite[Theorem 3.25]{Batkai-Piazzera} that $\mathcal{T}$ is given by
\begin{equation}
\mathcal{T}(t):=\left(\begin{array}{cc} V(t) & 0 \\ V_t & T_0(t) \end{array}\right),
\nonumber
\end{equation}
where $\big(T_0(t)\big)_{t\ge 0}$ is the left shift semigroup defined by
\begin{equation}
[T_0(t)f](\sigma):=
\begin{cases}
f(t+\sigma), & \quad\mbox{if}\quad \sigma\in[-1,-t), \\
0, & \quad\mbox{if}\quad \sigma\in[-t,0],
\end{cases}
\nonumber
\end{equation}
for all $f\in\mathrm{L}^p\big([-1,0],X\big)$, and $V_t$ is
\begin{equation}
(V_t x)(\sigma):=
\begin{cases}
V(t+\sigma)x, & \quad\mbox{if}\quad \sigma\in[-t,0], \\
0, & \quad\mbox{if}\quad \sigma\in[-1,-t).
\end{cases}
\nonumber
\end{equation}
for all $x\in X$. Since $\Phi$ is a bounded operator, $\mathcal{B}$ is also bounded on $\mathcal{E}$. Therefore, the semigroup $\mathcal{S}$ generated by $\mathcal{B}$ is
\begin{equation}
\mathcal{S}(t):=e^{t\mathcal{B}}=\mathcal{I}+t\mathcal{B}=\left(\begin{array}{cc} I & t\Phi \\ 0 & \widetilde I \end{array}\right),
\nonumber
\end{equation}
where $I$, $\widetilde I$, and $\mathcal{I}$ denote the identity operators on $X$, $\mathrm{L}^1\big([-1,0],X\big)$, and $\mathcal{E}$, respectively.

\noindent By formulae (\ref{spl_sq}), (\ref{spl_st}), and (\ref{spl_w}) of the sequential, Strang, and weighted splittings, the split solutions of the delay equation with initial value $\tbinom{x}{f}\in\mathcal{E}_p$ can be written as
\begin{align}
\label{sq-delay} \mathcal{U}^\mathrm{sq}_n(t)&= [\mathcal{S}(t/n)\mathcal{T}(t/n)]^n\tbinom{x}{f}, \medskip \\
\label{st-delay} \mathcal{U}^\mathrm{St}_n(t)&= [\mathcal{T}(t/2n)\mathcal{S}(t/n)\mathcal{T}(t/2n)]^n\tbinom{x}{f}, \medskip \\
\label{w-delay} \mathcal{U}^\mathrm{w}_n(t)&= [\Theta\mathcal{S}(t/n)\mathcal{T}(t/n)+(1-\Theta)\mathcal{T}(t/n)\mathcal{S}(t/n)]^n\tbinom{x}{f}
\end{align}
for $n\in\mathbb{N}$ fixed and $\Theta\in(0,1)$, for the sequential, Strang, and weighted splittings, respectively.
\begin{thm}[Theorem. 4.2, Corollary 4.3 in \cite{Csomos-Nickel}]
Under the Assumptions \ref{ass1}, the sequential, Strang, and weighted splittings applied to the delay equation (\ref{delay}) with sub-operators (\ref{matrices-delay}) are convergent at a fixed time level $t\ge 0$.
\label{thm:spl-for-delay}
\end{thm}
\noindent By Proposition \ref{prop:spl-conv-if-stab} we only have to show that the stability condition (\ref{stab_0}) is fulfilled. For the proof, see \cite{Csomos-Nickel}. (We note that the stability of the weighted splitting follows from the estimates in \cite{Csomos-Nickel} because we can choose $M=1$ for both semigroups, see Remark \ref{rem:stab_cond}.)


\noindent In Theorem \ref{thm:spl-for-delay} we showed that the sequential, Strang, and weighted splitting procedures are convergent when they are applied to the abstract Cauchy problem (\ref{acp_delay}) associated to the delay equation (\ref{delay}). We now combine these results with spatial approximations. Similarly as in Section \ref{sc:approx}, we define the following spaces and operators.
\begin{defn}
For $m\in\mathbb{N}$ we take
\renewcommand{\labelenumi}{(\roman{enumi})}
\begin{enumerate}
\item $X_m$ Banach spaces,
\item $P_m$ and $J_m$ operators satisfying the conditions in Definition \ref{def:proj} for the Banach spaces $X$ and $X_m$,
\item $\widetilde P_m$ and $\widetilde J_m$ operators between the Banach function spaces $\mathrm{L}^1\big([-1,0],X\big)$ and $\mathrm{L}^1\big([-1,0],X_m\big)$ defined as
\begin{equation}
(\widetilde P_mf)(\sigma):=P_mf(\sigma) \quad \mbox{and} \quad (\widetilde J_mf_m)(\sigma):=J_mf_m(\sigma)
\nonumber
\end{equation}
for all $f\in\mathrm{L}^1\big([-1,0],X\big)$ and $f_m\in\mathrm{L}^1\big([-1,0],X_m\big)$,
\item the spaces $\mathcal{E}:=X\times\mathrm{L}^1\big([-1,0],X\big)$ and $\mathcal{E}_m:=X\times\mathrm{L}^1\big([-1,0],X_m\big)$, and
\item the operators
\begin{equation}
\mathcal{P}_m:=\left(\begin{array}{cc} P_m & 0 \\ 0 & \widetilde P_m \end{array}\right) \quad \mbox{in } \mathcal{E} \quad \mbox{and} \quad
\mathcal{J}_m:=\left(\begin{array}{cc} J_m & 0 \\ 0 & \widetilde J_m \end{array}\right) \quad \mbox{in } \mathcal{E}_m.
\nonumber
\end{equation}
\end{enumerate}
\label{def:projekciok}
\end{defn}
\begin{prop}
The operators $\mathcal{P}_m,\mathcal{J}_m$ satisfy the approximating properties defined in Definition \ref{def:proj}.
\end{prop}

\begin{proof}
\begin{itemize}
\item $\|\mathcal{P}_m\|\le M_J$, since:
\begin{align}
\nonumber \|\mathcal{P}_m\|&=\max\limits_{\|x\|\le 1,\|f\|\le 1}\{\|P_mx\|,\|\widetilde P_mf\|\}
\le\max\limits_{\|x\|\le 1,\|f\|\le 1}\{\|P_mx\|,\sup\limits_{\sigma\in[-1,0]}\|P_mf(s)\|\} \\
\nonumber &\le\max\limits_{\|x\|\le 1,\|f\|\le 1} \{\|P_m\|\|x\|,\|P_m\|\sup\limits_{\sigma\in[-1,0]}\|f(s)\|\}
\le\max\{\|P_m\|,\|P_m\|\} \\
\nonumber &=\|P_m\|\le M_P,
\end{align}
and $\|\mathcal{J}_m\|\le M_J$ similarly.
\item $\mathcal{P}_m\mathcal{J}_m=\mathcal{I}_m$, where $\mathcal{I}_m$ denotes the identity operator on $\mathcal{E}_m$, since
\begin{align}
\nonumber \mathcal{P}_m\mathcal{J}_m
&=\left(\begin{array}{cc} P_m & 0 \\ 0 & \widetilde P_m \end{array}\right)
\left(\begin{array}{cc} J_m & 0 \\ 0 & \widetilde J_m \end{array}\right)
=\left(\begin{array}{cc} P_mJ_m & 0 \\ 0 & \widetilde P_m\widetilde J_m \end{array}\right),
\nonumber
\end{align}
where $P_mJ_m=I_m$ and $(\widetilde P_m\widetilde J_mf_m)(s)=P_mJ_mf_m(s)=f_m(s)$, hence $\widetilde P_m\widetilde J_m=\widetilde I_m$ in $\mathrm{L}^1\big([-1,0],X_m\big)$.
\item $\lim\limits_{m\to\infty}\mathcal{J}_m\mathcal{P}_m\tbinom{x}{f}=\tbinom{x}{f}$, because
\begin{equation}
\nonumber \left\| \binom{J_mP_mx-x}{\widetilde J_m\widetilde P_mf-f} \right\|
= \max\left\{ \|J_mP_mx-x\|,\|\widetilde J_m\widetilde P_mf-f\|_{\mathrm{L}^1} \right\},
\end{equation}
where $J_mP_mx$ converges to $x$ because of Definition \ref{def:proj}, and
\begin{align*}
&\|\widetilde J_m\widetilde P_mf-f\|_{\mathrm{L}^1}=\\
=&\int\limits_{-1}^0 \|(\widetilde J_m\widetilde P_m f-f)(\sigma)\|\mathrm{d}\sigma
=\int\limits_{-1}^0 \|J_mP_m f(\sigma)-f(\sigma)\|\mathrm{d}\sigma
\xrightarrow{m\to\infty}0,
\nonumber
\end{align*}
by Lebesgue's Theorem since it converges to zero for all $\sigma\in[-1,0]$ and $\widetilde J_m\widetilde P_m f-f$ is bounded:
\begin{align}
\nonumber & \|\widetilde J_m\widetilde P_m f-f\|_{\mathrm{L}^1}
\le \|\widetilde J_m\widetilde P_m f\|_{\mathrm{L}^1}+\|f\|_{\mathrm{L}^1}
\le \int\limits_{-1}^0 \|(\widetilde J_m\widetilde P_m f)(\sigma)\|\mathrm{d}\sigma +\|f\|_{\mathrm{L}^1} \\
\nonumber =& \int\limits_{-1}^0 \|J_mP_m f(\sigma)\|\mathrm{d}\sigma +\|f\|_{\mathrm{L}^1}
\le \|J_m\|\|P_m\|\int\limits_{-1}^0 \|f(\sigma)\|\mathrm{d}\sigma +\|f\|_{\mathrm{L}^1}  \\
\nonumber \le & M_JM_P\|f\|_{\mathrm{L}^1} +\|f\|_{\mathrm{L}^1}
= (M_JM_P+1)\|f\|_{\mathrm{L}^1}.
\end{align}
\end{itemize}
\end{proof}
From now on we assume the following, using the notations and terminology of the previous section (see Assumptions \ref{ass1}).
\begin{ass}
\mbox{}
\renewcommand{\labelenumi}{(\alph{enumi})}
\begin{enumerate}
\item There exist operators $C_m:X_m\to X_m$ generating strongly continuous semigroups $V_m$ satisfying $\|V_m(t)\|\leq 1$, such that $J_mC_mP_mx\to x$ for all $x\in X$.
\item There are bounded operators $\Phi_m: \mathrm{L}^1\big([-1,0],X_m\big) \to X_m$ satisfying $\|\Phi_m\|\leq c\|\Phi\|$, such that $\widetilde J_m \Phi_m \widetilde P_m f\to \Phi f$ for every $f\in \mathrm{L}^1\big([-1,0],X\big)$.
\end{enumerate}
\label{ass_approx_delay}
\end{ass}

\begin{rem}
A typical and usual choice for the operators $\Phi_m$ would be the one obtained through the spatial discretization as $\Phi_m f_m:=\widetilde P_m \Phi \widetilde J_m f_m$.
\end{rem}

We are now in the position to define the approximating operators for the delay equation. Consider the operators

\begin{equation}
\mathcal{A}_m:=\left(\begin{array}{cc} C_m & 0 \\ 0 & \frac{d}{d\sigma} \end{array}\right)
\end{equation}
on the domain
\begin{equation}
D(\mathcal{A}_m):=\left\{\tbinom{\xi_m}{\eta_m}\in D(C_m)\times \mathrm{W}^{1,1}\big([-1,0],X_m\big): \ \eta_m(0)=\xi_m  \right\},
\nonumber
\end{equation}
and the operators
\begin{equation*}
\mathcal{B}_m:=\left(\begin{array}{cc} 0 & \Phi_m \\ 0 & 0 \end{array}\right),  \quad D(\mathcal{B}_m):=\mathcal{E}_m,
\end{equation*}
corresponding to the spatially discretized (ordinary) delay equations
\begin{equation}
\left\{
\begin{aligned}
\frac{\mathrm{d}u^m(t)}{\mathrm{d}t}&=C_m u^m(t)+\Phi_m u^m_t, \qquad t\ge 0, \\
u^m(0)&=x_m\in X_m, \\
u^m_0&=f_m\in\mathrm{L}^1\big([-1,0],X_m\big).
\end{aligned}
\right.
\tag{DE$_{\text{m}}$}
\label{delay_m}
\end{equation}

\begin{thm}
Under these assumptions, the sequential, the Strang, and the weighted splittings are convergent for the delay semigroup.
\end{thm}

\begin{proof}
By Theorems \ref{thm:sq-conv-approx}, \ref{thm:st-conv-approx}, and \ref{thm:w-conv-approx},  we have to show that the stability and consistency conditions of Definition \ref{def:stab_conv} are satisfied for the operators $\mathcal{A}_m$, $\mathcal{B}_m$.

\noindent It is clear that the operators $\mathcal{A}_m$ generate strongly continuous semigroups
\begin{equation}
\mathcal{T}_m(t):=\left(\begin{array}{cc} V_m(t) & 0 \\ V_t^m & T_0(t) \end{array}\right),
\nonumber
\end{equation}
where
\begin{equation}
(V_t^m x)(\sigma):=
\begin{cases}
V_m(t+\sigma)x, & \quad\mbox{if}\quad \sigma\in[-t,0], \\
0, & \quad\mbox{if}\quad \sigma\in[-1,-t).
\end{cases}
\nonumber
\end{equation}
From the calculations in the proof of \cite[Theorem 4.2]{Csomos-Nickel}, especially formulae \cite{Csomos-Nickel}/(15) and \cite{Csomos-Nickel}/(16), we deduce
\begin{equation*}
\left\|\mathcal{T}_m(t)\right\| \leq 1+t,
\end{equation*}
and further, by the assumptions on the operators $\Phi_m$, $\mathcal{B}_m$ generate norm-continuous semigroups
\begin{equation}
\mathcal{S}_m(t):=\left(\begin{array}{cc} I_m & \Phi_m \\ 0 & I \end{array}\right)
\nonumber
\end{equation}
satisfying
\begin{equation*}
\left\|\mathcal{S}_m(t)\right\| \leq 1+tc\|\Phi\|.
\end{equation*}
Hence, the stability conditions (i)(a) and (b) from Definition \ref{def:stab_conv} are satisfied. Further, from the same arguments as in the proof of \cite[Theorem 4.2]{Csomos-Nickel} we obtain that the stability conditions \eqref{stab_approx} and \eqref{stab_approx_w} are satisfied with $M=1$ and $\omega=1+c\|\Phi\|$. Here we also have to use the fact that $t/n$ denotes the time step, and $k$ is the number of steps we have already done in order to reach the time level $t$. This means that the relation $k\le n$ always holds. Then we can apply the following estimate:
\begin{equation}
\|[\mathcal{S}_m(t/n)\mathcal{T}_m(t/n)]^k\| \le \|[\mathcal{S}_m(t/n)\mathcal{T}_m(t/n)]^n\|
\le \|\mathcal{S}_m(t/n)\|^n\|\mathcal{T}_m(t/n)\|^n \le \mathrm{e}^{t(1+c\|\Phi\|)}.
\nonumber
\end{equation}
Now we only have to check the consistency conditions. The condition (ii)(b) in Definition \ref{def:stab_conv} follows immediately form our Assumption \ref{ass_approx_delay} (b).

\noindent For condition (ii)(a), observe first that if $\tbinom{x}{f}\in D(\mathcal{A})$, then $f(0)=x$, implying $P_m f(0)=P_m x$, and so $\tbinom{P_m x}{\widetilde P_m f}\in D(\mathcal{A}_m)$. Hence, $\mathcal{P}_m D(\mathcal{A})\subset D(\mathcal{A}_m)$. Further, for $\tbinom{x}{f}\in D(\mathcal{A})$, we have that
\begin{equation*}
\mathcal{J}_m\mathcal{A}_m\mathcal{P}_m\tbinom{x}{f} = \mathcal{J}_m\mathcal{A}_m \tbinom{P_m x}{\widetilde P_m f}= \mathcal{J}_m \tbinom{C_mP_m x}{(\widetilde P_m f)'}=\tbinom{J_mC_mP_m x}{\widetilde J_m \widetilde P_m f'}.
\end{equation*}
Again, by our assumption on the operators $C_m$ and by Lebesgue's Theorem we see that
\begin{equation*}
\mathcal{J}_m\mathcal{A}_m\mathcal{P}_m\tbinom{x}{f} \to \tbinom{Cx}{f'}.
\end{equation*}
Hence, the desired stability and consistency conditions are satisfied meaning that the splittings are convergent.
\end{proof}


\section*{Acknowledgments}
This research was supported by the DAAD-PPP-Hungary Grant, project number D/05/01422. A. B. was supported by the J\'{a}nos Bolyai
Research Fellowship, the Bolyai-Kelly Scholarship, the OTKA Grant Nr. F049624, and by the Alexander von Humboldt-Stiftung.

The authors thank B\'{a}lint Farkas (Darmstadt) for helpful discussions and suggestions.

\bibliographystyle{plain}

\begin{thebibliography}{12}

\bibitem{Atkinson-Han}
K. Atkinson, W. Han,
\emph{Theoretical Numerical Analysis: A Functional Analysis Framework},
Springer--Verlag, New York. (2005)

\bibitem{Bagrinovskii-Godunov}
K. A. Bagrinovskii, S. K. Godunov,
Difference schemes for multidimensional problems (in Russian),
\emph{Dokl. Akad. Nauk. USSR} \textbf{115}, 431--433. (1957)

\bibitem{Botchev-Farago-Horvath}
M. A. Botchev, I. Farag\'{o}, R. Horv\'{a}th, Application of operator splitting to the Maxwell equations including a source term, \emph{Applied Numerical Mathematics} \textbf{59} (2009), 522-541.

\bibitem{Batkai-Piazzera}
A. B\'{a}tkai, S. Piazzera,
\emph{Semigroups for Delay Equations},
A K Peters, Wellesley, Massachusetts. (2005)

\bibitem{Bjorhus}
M.~Bj\o rhus,
Operator splitting for abstract Cauchy problems,
\emph{IMA J.~Numerical Analysis} \textbf{18}, 419--443. (1998)

\bibitem{Chernoff}
P. R. Chernoff,
Note on product formulas for operator semigroups,
\emph{J. Functional Analysis} \textbf{2}, 238--242. (1968)

\bibitem{Csomos-Farago}
P. Csom\'{o}s, I. Farag\'{o},
Error analysis of the numerical solution of split differential equations,
\emph{Mathematical and Computer Modelling}, \textbf{48}, 1090--1106. (2008)

\bibitem{Csomos-Nickel}
P. Csom\'{o}s, G. Nickel,
Operator splittings for delay equations,
\emph{Computers and Mathematics with Applications} \textbf{55}, 2234--2246. (2008)

\bibitem{Engel-Nagel}
K.--\,J. Engel, R. Nagel,
\emph{One-Parameter Semigroups for Linear Evolution Equations},
Springer--Verlag, Berlin. (2000)

\bibitem{Fabiano}
R. H. Fabiano,
Stability preserving Galerkin approximations for a boundary damped wave equation,
\emph{Nonlinear Anal., Theory Methods Appl.} \textbf{47}, 4545--4556. (2001)

\bibitem{Fabiano-Turi}
R. H. Fabiano, J. Turi,
Preservation of stability under approximation for a neutral FDE,
\emph{Dyn. Contin. Discrete Impulsive Syst.} \textbf{5}, 351--364. (1999)

\bibitem{Farago}
I. Farag\'{o},
Splitting methods and their application to the abstract Cauchy problems,
\emph{Lecture Notes in Computational Sciences} \textbf{3401}, Springer--Verlag 35--45. (2005)

\bibitem{Farago-Havasi}
I. Farag\'{o}, \'{A}. Havasi,
Consistency analysis of operator splitting methods for $C_0$-semigroups,
\emph{Semigroup Forum} \textbf{74}, 125--139. (2007)

\bibitem{Ichinose-etal}
T.~Ichinose, H.~Tamura, H.~Tamura, and V.~A.~Zagrebnov,
Note on the Paper ``The norm convergence of the Trotter\,--\,Lie product formula with error bound'' by Ihinose and Tamura,
\emph{Communications in Mathematical Physics} \textbf{221}, 499--510. (2001)

\bibitem{Ito-Kappel}
K. Ito, F. Kappel,
\emph{Evolution Equations and Approximations},
World Scientific, River Edge, N. J. (2002)

\bibitem{Ito-Kappel2}
K. Ito, F. Kappel,
The Trotter\,--\,Kato theorem and approximation of PDE's,
\emph{Mathematics of Computation} \textbf{67}, 21--44. (1998)

\bibitem{Kappel}
F. Kappel,
Semigroups and delay equations, in: H. Brezis et al. (Eds.) ``Semigroups, Theory and Application'' Vol. II.,
\emph{Pitman Research Notes in Mathematics} \textbf{152}, Longman, 136--176. (1986)

\bibitem{Larsson-Thomee}
S. Larsson, V. Thom\'{e}e,
\emph{Partial Differential Equations with Numerical Methods},
Springer--Verlag, Heidelberg. (2003)

\bibitem{Lax}
P. D. Lax,
\emph{Functional Analysis},
John Wiley \& Sons, New York. (2002)

\bibitem{Marchuk}
G. I. Marchuk,
Some application of splitting-up methods to the solution of mathematical physics problems,
\emph{Applik. Mat.} \textbf{13}, 103--132. (1968)

\bibitem{Neidhardt-Zagrebnov}
H.~Neidhardt, V.~A.~Zagrebnov,
On error estimates for the Trotter\,--\,Kato product formula,
\emph{Letters in Mathematical Physics} \textbf{44}, 169--186. (1998)

\bibitem{Neidhardt-Zagrebnov2}
H.~Neidhardt, V.~A.~Zagrebnov,
Trotter\,--\,Kato product formula and operator-norm convergence,
\emph{Communications in Mathematical Physics} \textbf{205}, 129--159. (1999)

\bibitem{Pazy}
A. Pazy,
\emph{Semigroups of Linear Operators and Applications to Partial Differential Equations},
Springer--Verlag, New York. (1983)

\bibitem{Richtmyer-Morton}
R. Richtmyer, K. W. Morton,
\emph{Difference Methods for Initial Value Problems}.
Krieger Publishing, Malabar. (1994)

\bibitem{Strang}
G. Strang,
On the construction and comparison of difference schemes,
\emph{SIAM J. Numerical Anaysis.}, \textbf{5}, 506--517. (1968)

\bibitem{Strang2}
G. Strang,
Approximation of semigroups and the consistency of different schemes,
\emph{Proc. Amer. Math. Soc.} \textbf{20}, 1--7. (1969)

\bibitem{Zagrebnov}
V.~A.~Zagrebnov,
\emph{Topics in the Theory of Gibbs Semigroups}, Leuven University Press. (2003)

\bibitem{Zlatev}
Z. Zlatev,
\emph{Computer treatment of large air pollution models},
KLUWER Academic Publishers, Dordrecht, 1995.

\bibitem{Zlatev-Dimov}
Z. Zlatev, I. Dimov,
\emph{Computational and Numerical Challenges in Environmental Modelling},
Elsevier, Amsterdam, 2006.

\end{thebibliography}


\end{document}